\documentclass[11pt]{article}

\usepackage{amsmath,amsthm,amssymb,mathrsfs}
\usepackage{epsfig}
\usepackage{graphics}
\usepackage{colordvi}

\font\tencmmib=cmmib10 \skewchar\tencmmib '60
\newfam\cmmibfam
\textfont\cmmibfam=\tencmmib

\def\lessim{\ \lower4pt\hbox{$
\buildrel{\displaystyle <}\over\sim$}\ }
\def\gessim{\ \lower4pt\hbox{$\buildrel{\displaystyle >}
\over\sim$}\ }

\def\si{{\sigma}}

\newcommand{\la}{\langle}
\newcommand{\ra}{\rangle}

\newcommand{\e}{\mathbb{E}}
\newcommand{\p}{\mathbb{P}}

\newcommand{\s}{\boldsymbol{\sigma}}

\begin{document}
\newtheorem{Theorem}{Theorem}
\newtheorem{Proposition}{Proposition}
\newtheorem{Lemma}{Lemma}
\newtheorem{Corollary}{Corollary}
\newtheorem{Definition}{Definition}
\newtheorem{Assumption}{Assumption}
\newtheorem{Example}{Example}

\numberwithin{Theorem}{section} \numberwithin{Lemma}{section}
\numberwithin{Proposition}{section} \numberwithin{equation}{section}
\numberwithin{Definition}{section} \numberwithin{Corollary}{section}
\numberwithin{Example}{section}

\title{Disorder chaos in the spherical mean-field model}

\author{
Wei-Kuo Chen\thanks{Department of Mathematics, University of Chicago. Email: wkchen@math.uchicago.edu} 
\and
Hsi-Wei Hsieh\thanks{Institute of Mathematics, Academia Sinica. Email: c76068@math.sinica.edu.tw}
\and
Chii-Ruey Hwang\thanks{Institute of Mathematics, Academia Sinica. Email: crhwang@sinica.edu.tw}
\and
Yuan-Chung Sheu\thanks{Department of Applied Mathematics, National Chiao Tung University. Email: sheu@math.nctu.edu.tw}
}
\date{}
\maketitle
\begin{abstract}
We consider the problem of disorder chaos in the spherical mean-field model. It is concerned about the behavior of the overlap between two independently sampled spin configurations from two Gibbs measures with the same external parameters. The prediction states that if the disorders in the Hamiltonians are slightly decoupled, then the overlap will be concentrated near a constant value. Following Guerra's replica symmetry breaking scheme, we establish this at the level of the free energy as well as the Gibbs measure irrespective the presence or absence of the external field. 
\end{abstract}


%60G99

%{\it Running Title}: Disorder chaos in the spherical SK model

\maketitle
%\newpage

\section{Introduction and main results}
This paper is concerned about the chaos problem in mean-field spin glasses. It arose from the discovery that in some models, a small perturbation to the external parameters will result in a dramatic change to the overall energy landscape and the organization of the pure states of the Gibbs measure. Over the past decades, physicists have intensively studied chaos phenomenon at the free energy level utilizing the replica method, where most related works were discussed in models with Ising spin. We refer readers to the survey of Rizzo \cite{rizzo07} and the references therein along this direction. Recently, mathematical results also have been obtained in the Ising-spin mixed even-spin model. Chaos in disorder without external field was considered in Chatterjee \cite{chatterjee2009disorder} and a more general situation with external field was handled in Chen \cite{chen2013disorder}. Some special cases of temperature chaos were obtained in Chen and Panchenko \cite{chen2014chaos} and Chen \cite{chen2013approach}. 

The aim of this work is to investigate the problem of disorder chaos in the spherical mean-field model. Our approach is based on Guerra's replica symmetry breaking bound for the coupled free energy with overlap constraint. This methodology was adapted in Chen \cite{chen2013disorder} to establish chaos in disorder for Ising-spin mixed even-spin model with external field, where many estimates were highly involved due to the nature of the Ising spin. In this paper, we first want to illustrate how the same method may as well be applied to the spherical model and clarify several ideas behind the proof sketch of Research Problem 15.7.14. about disorder chaos problem in Talagrand \cite{talagrand2011meanII} and Chen \cite{chen2013disorder} with more explicit and simpler computations. Our results cover both situations when the external field is present or absent. On the technical ground, we intend to understand to what extent the current approach can reach. In Panchenko and Talagrand \cite{panchenko2007overlap}, the same approach as the present paper was formerly used to study the conjectures of ultrametricity and chaos in temperature for spherical pure even-spin model, where it has been pointed out that these problems can not be achieved at the level of the free energy. We show that chaos in disorder is indeed a much stronger effect and can still be established at the free energy level even in the mixed even-spin model. 

We now state our main results. For each $N\in\mathbb{N}$, let $X_N$ be a centered Gaussian process indexed by the configuration space$$S_{N}=\Bigl\{\boldsymbol{\sigma}=(\sigma_{1}, \dots, \sigma_{N})\in\mathbb{R}^{N}: \sum_{i\leq N}\sigma_{i}^{2}=N\Bigr\}$$ 
and equipped with the covariance structure
\begin{align*}
\e X_N(\s^1)X_N(\s^2)=N\xi(R_{1,2}),
\end{align*}
where $R_{1, 2}=N^{-1}\boldsymbol{\sigma}^{1}\cdot\boldsymbol{\sigma}^{2}$ is called the overlap between two configurations $\boldsymbol{\sigma}^{1},\boldsymbol{\sigma}^{2}\in S_N$ and $\xi:[0,1]\rightarrow \mathbb{R}$ is an even convex function with $\xi''(x)>0$ for $x>0$ and $\xi'''(x)\geq 0$ for $x\geq 0.$ The spherical model is defined on $S_N$ and its Hamiltonian takes the form,
\begin{align*}
-H_N(\s)=X_N(\s)+h\sum_{i=1}^N\sigma_i.
\end{align*}
Set the corresponding Gibbs measure,
\begin{align*}
dG_N(\s)=\frac{1}{Z_N}\exp\Bigl(-H_N(\s)\Bigr)d\lambda_N(\s),
\end{align*}
where $d\lambda_N$ is the uniform probability measure on $S_N$ and the normalizing factor $Z_N$ is called the partition function. An important example of $\xi$ is the mixed even-spin model, $\xi(x)=\sum_{p\geq 1}\beta_{p}^2x^{2p}$ for some sequence of real numbers $(\beta_{p})_{p\geq 1}$ with $\sum_{p\geq 1}2^p\beta_p^2<\infty.$ Denote by $p_N=N^{-1}\e \log Z_N$ the limiting free energy.  Probably the most important fact about the spherical model is the Crisanti-Sommers formula \cite{crisanti1993sphericalp} for the limiting free energy,
\begin{align}\label{pf}
\lim_{N\rightarrow\infty}p_N=\inf_{x,b}\mathcal{P}(x,b).
\end{align}
Here for any distribution function $x$ on $[0,1]$ and $b>\max\bigl\{1,\int_0^1\xi''(s)x(s)ds\bigr\}$, 
\begin{align}\label{parisifunc}
 \mathcal P(x,b):= \frac{1}{2} \left( \frac{h^2}{b-d(0)} + \int_0^1 \frac{\xi''(q)}{b-d(q)} dq +b-1- \log b -\int_0^1 \xi''(q)x(q)dq \right),
\end{align}
where $d(q):= \int_q^1 \xi''(s) x(s) ds$. The formula \eqref{pf} was firstly verified by Talagrand \cite{talagrand2006free} and later generalized to the spherical mixed $p$-spin model including odd $p$ in Chen \cite{Chen12}. A key fact of the variational formula \eqref{pf} is the existence and uniqueness of the optimizer or the functional order parameter, which is guaranteed by Talagrand \cite[Theorem 1.2]{talagrand2006free}.

In the problem of disorder chaos, we are interested in understanding how the system would behave when the disorder is perturbed. To this end, we shall consider two copies $X_N^1$ and $X_N^2$ of $X_N$ with covariance
\begin{align*}
\mathbb{E}X_{N}^1(\boldsymbol{\sigma}^{1})X_{N}^2(\boldsymbol{\sigma}^{2})&=tN\xi(R_{1,2})
\end{align*}
for some $t\in [0,1].$ In the same manner as $H_N$, $G_N$ and $Z_N$, we denote by $H_N^1,H_N^2$ the Hamiltonians, $G_N^1,G_N^2$ the Gibbs measures and $Z_N^1,Z_N^2$ the partition functions corresponding to $(X_N^1,h)$ and $(X_N^2,h)$, respectively.
Let $\left<\cdot\right>$ denote the Gibbs expectation with respect to the product measure $dG_N^1(\boldsymbol{\sigma}^1)\times dG_N^2(\boldsymbol{\sigma}^2).$ If $t=1,$ these two systems are identically the same in which case the limiting distribution of the overlap $R_{1,2}$ under the measure $\mathbb{E}\left<\cdot\right>$ is typically non-trivial in the replica symmetry breaking region. Contrary to the situation $t=1$, our main results on disorder chaos stated in the following theorems say that the system will change dramatically at the level of the free energy and the Gibbs measure if the two systems are decoupled $0<t<1.$

\begin{Theorem}\label{thm0}
For $u\in [-1,1]$ and $\alpha>0,$ define the coupled partition function,
$$
Z_{N,u,\alpha}=\int_{|R_{1,2}-u|<\alpha}\exp\bigl(-H_N^1(\s^1)-H_N^2(\s^2)\bigr)d\lambda_N(\s^1)d\lambda_N(\s^2).
$$
and set the coupled free energy,
\begin{align}\label{thm0:eq2}
p_{N,u,\alpha}=\frac{1}{N}\e\log Z_{N,u,\alpha}.
\end{align}
 If $0<t<1$, there exists some $u^*\in [0,1)$ such that for all $u\neq u^*$, 
\begin{align} \label{thm0:eq1}
\limsup_{\alpha\downarrow 0}\limsup_{N\rightarrow\infty}p_{N,u,\alpha}<2\inf_{x,b}\mathcal{P}(x,b).
\end{align}
\end{Theorem}

In other words, there is free energy cost if $u\neq u^*$ for $0<t<1$. Here the determination of $u^*$ is a technical issue, which is described through an equation related to the Crisanti-Sommers formula as well as the associated optimizer. We shall leave the details to Section $3$. Roughly speaking, $u^*$ is equal to zero if $h=0$ and it stays positive if $h\neq 0$. As an immediate application of the Gaussian concentration of measure, Theorem \ref{thm0} yields the concentration of the overlap near the constant $u^*$ in the following theorem.
 
\begin{Theorem}\label{thm01}
If $0<t<1$, then there exists some $u^*\in [0,1)$ such that for any $\varepsilon>0,$ 
\begin{align} \label{expoinq}
\mathbb{E} \langle \mathbf{1}_{  \{ |R_{1,2}-u^* |> \varepsilon \} } \rangle \leq K \exp \left(   -\frac{N}{K}\right)
\end{align}
for all $N\geq 1,$ where $K$ is a constant independent of $N$.
\end{Theorem}

This paper is organized as follows. Our approach is based on a two-dimensional extension of the Guerra replica symmetry breaking bound for \eqref{thm0:eq2} and a sketch of the proof for disorder chaos in the Ising-spin mixed even-spin model as was outlined in Talagrand \cite[Section 15.7]{talagrand2011meanII} and later implemented in Chen \cite{chen2013disorder}. In Section 2, using Guerra's bound, we will compute explicitly manageable upper bounds for the coupled free energy \eqref{thm0:eq2}. These results will be used in Section 3. We first describe how to determine the constant $u^*$ and then conclude Theorem \ref{thm0}. Finally, we carry out the proof of Theorem \ref{thm01}.

%%%%%%%%%%%%%%%%%%%%%%%%%%%%%%%%%%%%%%%%%%%%%%%%%%%%%%%%%%%

\section{Guerra's interpolation}

The main goal of this section is to derive the following upper bound for the coupled free energy \eqref{thm0:eq2}, which is an extended version of Proposition 7.8 in \cite{talagrand2006free}.  

\begin{Proposition}\label{prop04}
For any distribution function $x$ on $[0,1]$, $\lambda \in \mathbb{R}$ and $b>\int_0^1\xi''(s)x(s)ds+|\lambda |$, we have that for any $u\in [-1,1],$
\begin{align} \label{upositive}
\limsup_{\alpha \downarrow 0} \limsup_{N \rightarrow \infty} p_{N,u,\alpha} &\leq  \mathcal{P}_u(x,b,\lambda),
\end{align}
where the functional $\mathcal{P}_u(x,b,\lambda)$ is defined as follows. Set $d(q)=\int_q^1 \xi''(s)x(s)ds$ and $$\phi_u(q)=d(u)+\frac{1-t}{1+t}(d(q)-d(u)).$$ Define
\begin{align*}
\mathcal{P}_u(x,b,\lambda) &:=\left\{
\begin{array}{ll}
T_u(x,b,\lambda)+\frac{h^2}{b-\lambda- d(0)},&\mbox{if $u\in[0,1]$},\\
T_u(x,b,\lambda)+\frac{h^2}{b-\lambda-\phi_{|u|}(0)},&\mbox{if $u\in[-1,0])$},
\end{array}\right.
\end{align*}
where 
\begin{align}
\begin{split}\label{eq12}
T_u(x,b,\lambda)&=\log \sqrt{\frac{b^2}{b^2-\lambda^2}}+\frac{1+t}{2}  \int_0^{|u|} \frac{\xi''(s)}{b-\eta\lambda-d(s)} ds+ \frac{1-t}{2}  \int_0^{|u|} \frac{\xi''(s)}{b+\eta\lambda-\phi_{|u|}(s)} ds\\
&\quad +\frac{1}{2} \int_{|u|}^1 \frac{\xi''(s)}{b-\lambda-d(s)} ds + \frac{1}{2} \int_{|u|}^1 \frac{\xi''(s)}{b+\lambda-d(s)} ds \\
 &\quad-\lambda u +b-1- \log b - \int_0^1 \xi''(q) x(q)dq.
 \end{split}
\end{align}

\end{Proposition}

We mainly follow the procedure of the proof of Theorem 5.3 in \cite{talagrand2006free} to prove Proposition \ref{prop04}. Fix $u \in [-1,1]$ and $\eta \in \{ 1,-1\}$ with $u = \eta |u|$. It suffices to prove \eqref{upositive} only for discrete $x.$ For $k\geq 0$, consider two sequences of real numbers $\mathbf{m}=(m_\ell)_{0\leq \ell\leq k}$ and $\mathbf{q}=(q_{\ell})_{0\leq \ell\leq k+1}$ that satisfy
\begin{align}
\begin{split}\label{eq_kmq}
&0=m_0\leq m_1\leq \cdots\leq m_k\leq m_{k}=1,\\
&0=q_0\leq q_1\leq \cdots\leq q_{k+1}\leq q_{k+1}=1.
\end{split}
\end{align}
Let $x$ be a distribution function on $[0,1]$ associated to this triplet $(k,\mathbf{m},\mathbf{q})$, that is, $x(q)=m_\ell$ for $q\in[q_\ell,q_{\ell+1})$ and $0\leq \ell\leq k$ and $x(1)=1.$ Without loss of generality, we may assume that $q_\tau=|u|$ for some $0\leq \tau\leq k+1$. Define the sequence $\boldsymbol{n}=(n_p)_{0\leq p\leq k}$ by
\begin{align}
\begin{split}\label{eq1}
&0=n_0,n_1=\frac{m_1}{1+t},\ldots,n_{\tau-1}=\frac{m_{\tau-1}}{1+t},n_{\tau}=m_{\tau},\ldots,n_{k}=m_k.
\end{split}
\end{align}
%where $m$ is a number satisfying $\frac{m_{\tau-1}}{1+t}\leq m \leq m_\tau$. % 
We consider further independent pairs of centered Gaussian random vectors $(y_p^1,y_p^2)_{0\leq p \leq k}$ that possess covariance
\begin{align}
\begin{split}
\label{eq3}
&\mathbb{E}(y_p^j)^2 = \xi'(q_{p+1})-\xi'(q_p), \ \ 0\leq p \leq k, \ j=1,2, \\
&\mathbb{E}y_p^1 y_p^2 = \eta t\xi'((q_{p+1})-\xi'(q_p)), \ \ 0\leq p <\tau, \\
&y_p^1, y_p^2 \ \textrm{are independent}, \ \  \tau\leq p \leq k.
\end{split}
\end{align}
Let $(y_{i,p}^1,y_{i,p}^2)_{0 \leq p \leq k}$ be independent copies of $(y_p^1,y_p^2)_{0 \leq p \leq k}$ for $1 \leq i\leq N$ and be independent of $X_N^1,X_N^2$. Following Guerra's scheme, we define the interpolated Hamiltonian $H_{N,a}(\s^1,\s^2)$ for $a\in [0,1]$,
$$
-H_{N,a}(\s^1,\s^2)=\sqrt{a}(X_N^1(\s^1)+X_N^2(\s^2))+\sum_{j=1}^2\sum_{i=1}^N\Bigl(\sqrt{1-a}\sum_{0\leq p\leq k}y_{i,p}^j+h\Bigr)\sigma_i^j.
$$
Define
\begin{align*}
F_{k+1}(a) = \log \int_{|R_{1,2}-u|< \alpha} \exp \bigl(-H_{N,t}(\s^1,\s^2)\bigr)  d \lambda_N (\boldsymbol{\sigma}^1) d \lambda_N (\boldsymbol{\sigma}^2).
\end{align*}
Denote by $\mathbb{E}_{p}$ the expectation in the random variables $(y_{i,p}^1,y_{i,p}^2),\ldots,(y_{i,k}^1,y_{i,k}^2)$ and define recursively for $0\leq p \leq k$,
\begin{align*}
F_{p}(a) = \left\{
\begin{array}{ll}
   \frac{1}{n_{p}} \log \mathbb{E}_{p} \exp n_{p} F_{p+1}(a),         & \textrm{if} \ n_{p} \neq 0, \\
   \\
    \mathbb{E}_{p} F_{p+1}(a),         & \textrm{if} \ n_{p}=0.
\end{array} \right.
\end{align*}
Finally set $\phi(a)=N^{-1}\e F_0(a)$ and denote $F_0=\phi(0).$ Following essentially the same proof as either Theorem 5 in \cite{panchenko2007overlap} or Theorem 7.1 in \cite{talagrand2006free}, one can prove that the interpolated free energy $\phi$ yields

\begin{Proposition}\label{RSB}
For any $\alpha>0$ and $x$ corresponding to $(k,\mathbf{m},\mathbf{q})$, we have that
\begin{align}\label{RSB:eq1}
p_{N,u,\alpha}&\leq F_0-(1+t)\sum_{0\leq p \leq \tau} n_p(\theta(\rho_{p+1})-\theta(\rho_p))- \sum_{\tau <p \leq k} n_p(\theta(\rho_{p+1})-\theta(\rho_p) ) + \mathcal{R}
\end{align}
where $\theta(q)=q\xi'(q)-\xi(q)$ and $\limsup_{N\rightarrow\infty}|\mathcal{R}|=0.$
\end{Proposition}

Substituting \eqref{eq1} into the right-hand side of \eqref{RSB:eq1}, a direct computation gives
\begin{align}\label{eq7}
&(1+t)\sum_{0\leq p < \tau} n_p(\theta(q_{p+1})-\theta(q_p))+\sum_{\tau \leq p \leq k+1} n_p(\theta(q_{p+1})-\theta(q_p) )\notag\\
&=\sum_{0\leq p\leq k}m_p(\theta(q_{p+1})-\theta(q_p))\notag\\
&=\int_0^1 \xi''(q)x(q)dq.
\end{align}
We now turn to the control of the quantity $F_0.$ For $b>1$, we denote by $\nu_N^b$ the probability measure of $N$ i.i.d. Gaussian random variables with mean zero and variance $b^{-1}$, that is,
\begin{align*}
d\nu_N^b(\boldsymbol{y})=\left( \frac{b}{2 \pi} \right)^{\frac{N}{2}} \exp \left( -\frac{b}{2} \| \boldsymbol{y} \|^2 \right)d\boldsymbol{y}.
\end{align*}
Let $\tau_N^b = -N^{-1} \log \nu_N^b (\{\boldsymbol{\sigma}: \| \boldsymbol{\sigma} \|^2 \geq N  \})$. Without ambiguity, we simply write $\nu^b$ for $\nu_1^b$. Given a number $\lambda$, we define the function
\begin{align*}
B_{k+1}(\boldsymbol{x}^1,\boldsymbol{x}^2,\lambda) = \log \int \exp(\boldsymbol{x}^1 \cdot \boldsymbol{\sigma}^1 + \boldsymbol{x}^2 \cdot \boldsymbol{\sigma}^2 + \lambda \boldsymbol{\sigma}^1 \cdot \boldsymbol{\sigma}^2) d \nu_N^b (\boldsymbol{\sigma}^1) d \nu_N^b (\boldsymbol{\sigma}^2)
\end{align*}
for $\boldsymbol{x}^1,\boldsymbol{x}^2\in\mathbb{R}^N$ and recursively, for $1 \leq p\leq k$,
\begin{align*}
B_{p}(\boldsymbol{x}^1,\boldsymbol{x}^2,\lambda) = \left\{
\begin{array}{ll}
 \frac{1}{n_{p}} \log \mathbb{E} \exp n_{p}B_{p+1} (\boldsymbol{x}^1+\boldsymbol{y}_{p}^1,\boldsymbol{x}^2+\boldsymbol{y}_{p}^2, \lambda), &
\textrm{if} \ n_{p} \neq 0, \\
\\
\mathbb{E} B_{p+1} (\boldsymbol{x}^1+\boldsymbol{y}_{p}^1,\boldsymbol{x}^2+\boldsymbol{y}_{p}^2, \lambda), & \textrm{if} \ n_{p}=0,
\end{array} \right.
\end{align*}
where $\boldsymbol{y}_p^j = (y_{i,p}^j)_{1 \leq i \leq N}$ for $j=1,2$ and $0\leq p \leq k$. Let $\boldsymbol{h}=(h,\ldots,h).$ Following the same argument as in the proof of Lemma 7.1 \cite{talagrand2006free}, we obtain 

\begin{Lemma}\label{lem1} Let $u\in[-1,1]$, $\alpha>0$ and $\lambda\in\mathbb{R}.$ If $b>\int_0^1\xi''(s)x(s)ds+|\lambda|,$ then
\begin{align}\label{lem1:eq1}
F_0&\leq -\lambda u+|\lambda|\alpha+2\tau_N^b+\frac{1}{N}\e B_1(\boldsymbol{h}+\boldsymbol{y}_0^1,\boldsymbol{h}+\boldsymbol{y}_0^2,\lambda).
\end{align}
\end{Lemma}

To compute the term $N^{-1}\e B_1(\boldsymbol{h}+\boldsymbol{y}_0^1,\boldsymbol{h}+\boldsymbol{y}_0^2,\lambda)$, we will need a technical lemma.

\begin{Lemma}\label{prop05}
For $x^1$, $x^2 \in \mathbb{R}$, we define
\begin{align}
\begin{split}\label{prop05:eq1}
 J^1_{k+1}(x^1,x^2,\lambda)&= \log \int \exp \rho^1 \left( \frac{x^1+x^2}{\sqrt{2}} \right)  d \nu_1^{b-\lambda} (\rho^1) = \frac{(\frac{x^1+x^2}{\sqrt{2}})^2}{2(b-\lambda)}, \\
 J^2_{k+1}(x^1,x^2,\lambda)&=\log \int \exp \rho^2 \left( \frac{x^1-x^2}{\sqrt{2}} \right)  d \nu_1^{b+\lambda} (\rho^2) =  \frac{(\frac{x^1-x^2}{\sqrt{2}})^2}{2(b+\lambda)},
 \end{split}
\end{align}
and recursively for $1\leq p\leq k$ and $j=1,2$,
\begin{align*}
&J^j_{p}(x^1,x^2,\lambda)= \left\{
\begin{array}{ll}
\frac{1}{n_{p}} \log \mathbb{E} \exp n_{p}J^j_{p+1}(x^1+y_{p}^1,x^2+y_{p}^2,\lambda),   & \textrm{if} \ n_{p} \neq 0, \\
\\
\mathbb{E} J^j_{p+1}(x^1+y_{p}^1,x^2+y_{p}^2,\lambda),     & \textrm{if} \ n_{p} = 0,
\end{array} \right.
\end{align*}
then 
\begin{align} 
\begin{split}\label{Bequa}
&\frac{1}{N}\mathbb{E}B_1(\boldsymbol{h}+\boldsymbol{y}_0^1,\boldsymbol{h}+\boldsymbol{y}_0^2,\lambda)\\
& = \log \sqrt{\frac{b^2}{b^2-\lambda^2}}+\mathbb{E} J_{1}^1(h+y_0^1,h+y_0^2,\lambda)+ \mathbb{E} J_{1}^2(h+y_0^1,h+y_0^2,\lambda).
\end{split}
\end{align}
\end{Lemma}

\begin{proof} For $x^1$, $x^2 \in \mathbb{R}$ and $1 \leq p\leq k+1$, we define the following functions
\begin{align*}
\Gamma_{k+1}(x^1,x^2,\lambda)& = \log \int \exp \left(x^1 \sigma^1+ x^2 \sigma^2 +\lambda \sigma^1 \sigma^2  \right)
 d \nu^b (\sigma^1) d \nu^b(\sigma^2), \\
 \Gamma_{p}(x^1,x^2,\lambda)& =\frac{1}{n_{p}} \log \mathbb{E} \exp n_{p} \Gamma_{p+1}(x^1+y_{p}^1,x^2+y_{p}^2,\lambda),  \ 1 \leq p\leq k.
\end{align*}
Since $(\si_1^1,\si_1^2),\ldots,(\si_N^1,\si_N^2)$ are independent under the measure $\nu_N^b\times \nu_N^b$, we see recursively that
$B_{p}(\boldsymbol{x}^1, \boldsymbol{x}^2,\lambda) = \sum_{i \leq N} \Gamma_{p}(x^1_i,x^2_i,\lambda),$ where $\boldsymbol{x}^j=(x_i^j)_{i\leq N}$ for $j=1,2$. Consequently,
\begin{align*}
&\frac{1}{N} \mathbb{E}B_1(\boldsymbol{h}+\boldsymbol{y}_0^1,\boldsymbol{h}+\boldsymbol{y}_0^2,\lambda) = \mathbb{E} \Gamma_1(h+y_0^1,h+y_0^2,\lambda).
\end{align*}
Now making change of variables 
\begin{align*}
\sigma^1 = \frac{\rho^1+\rho^2}{\sqrt{2}}, \ \sigma^2 = \frac{\rho^1-\rho^2}{\sqrt{2}}
\end{align*}
and noting that $\rho^1,\rho^2$ are i.i.d. Gaussian with mean zero and variance $b^{-1}$, we obtain
\begin{align*}
&\Gamma_{k+1}(x^1,x^2,\lambda) \\
&= \log \int \exp \left(\frac{x^1+x^2}{\sqrt{2}} \rho^1 + \frac{x^1-x^2}{\sqrt{2}} \rho^2 + \lambda \frac{(\rho^1)^2}{2} -  \lambda \frac{(\rho^2)^2}{2} \right) d \nu^b(\rho^1)d \nu^b(\rho^2) \\
&= \log \int \exp \left(\frac{x^1+x^2}{\sqrt{2}} \rho^1 + \lambda \frac{(\rho^1)^2}{2} \right) d \nu^b(\rho^1) +\log \int \exp \left(  \frac{x^1-x^2}{\sqrt{2}} \rho^2  -  \lambda \frac{(\rho^2)^2}{2} \right) d  \nu^b(\rho^2) \\
&= \log \sqrt{\frac{b^2}{b^2-\lambda^2}} + J^1_{k+1}(x^1,x^2,\lambda) + J^2_{k+1}(x^1,x^2,\lambda).
\end{align*}
Since $y^1_{p}+y^2_{p}$ and $y^1_{p}-y^2_{p}$ are independent, starting with \eqref{prop05:eq1}, an iterative argument implies that $J^1_{p+1}(x^1+y_{p}^1,x^2+y_p^2)$ and $J^2_{p+1}(x^1+y_{p}^1,x^2+y_p^2)$ are independent of each other, which yields
\begin{align*}
 \Gamma_{p}(x^1,x^2,\lambda)=\log \sqrt{\frac{b^2}{b^2-\lambda^2}} + J_{p}^1(x^1,x^2,\lambda)+J_{p}^2(x^1,x^2,\lambda)
\end{align*}
for $1\leq p\leq k+1$ and hence \eqref{Bequa}.
\end{proof}

\begin{proof}[Proof of Proposition \ref{prop04}] The proof is essentially based on an explicit calculation of the right hand-side of \eqref{Bequa}. To lighten notations, we set
\begin{align*}
v_p&=\xi'(q_{p+1})-\xi'(q_p)\,\,\mbox{if $0\leq p\leq k$,}\\
d_p'&=\sum_{p\leq \ell \leq \tau-1}n_\ell v_\ell=\frac{1}{1+t}\int_{q_p}^{q_\tau}\xi''(s)x(s)ds\,\,\mbox{if $0\leq p\leq \tau-1$},\,\,d_{\tau}'=0,\\
d_p&= \sum_{p\leq \ell \leq k}n_\ell v_\ell=\int_{q_p}^1\xi''(s)x(s)ds\,\,\mbox{if $\tau\leq p\leq k$},\,\,d_{k+1}=0.
\end{align*}
Recall \eqref{eq3}. It is straightforward to obtain that for $\tau\leq p\leq k$
\begin{align*}
&\mathbb{E} \left( \frac{y_p^1+y_p^2}{\sqrt{2}} \right)^2 =\mathbb{E} \left( \frac{y_p^1-y_p^2}{\sqrt{2}} \right)^2 = v_p
\end{align*}
 and for $0\leq p<\tau,$
\begin{align*}
&\mathbb{E} \left( \frac{y_p^1+y_p^2}{\sqrt{2}} \right)^2 = (1+\eta t)v_p, \,\,\mathbb{E} \left( \frac{y_p^1-y_p^2}{\sqrt{2}} \right)^2 = (1-\eta t)v_p.
\end{align*}
Combining these with the formula that for a standard Gaussian random variable $z$, $nv<L$ and $y\in\mathbb{R}$, 
\begin{align*}
&\frac{1}{n}\log\e\exp \frac{n}{2L}(y+\sqrt{v}z)^2=\left\{\begin{array}{ll}
\frac{y^2}{2(L-nv)}+\frac{1}{2n}\log\frac{L}{L-nv},&\mbox{if $n>0$},\\
\frac{y^2}{2L}+\frac{v}{2L},&\mbox{if $n=0$,}
\end{array}
\right.
\end{align*}
an iterative procedure leads to
\begin{align} 
\begin{split}\label{J1g}
\mathbb{E} J_1^1(h+y_0^1,h+y_0^2,\lambda)
&=\frac{2h^2}{2(b-\lambda-(d_{\tau}+(1+\eta t)d_0'))}\\
&+\frac{1}{2}\sum_{p=0}^{\tau-1}\frac{1}{n_p}\log\frac{b-\lambda-(d_{\tau}+(1+\eta t)d_{p+1}')}{b-\lambda-(d_{\tau}+(1+\eta t)d_p')}\\
&+\frac{1}{2}\sum_{p=\tau}^{k}\frac{1}{n_p}\log\frac{b-\lambda-d_{p+1}}{b-\lambda-d_p}
  \end{split}
\end{align}
and
\begin{align}
\begin{split} \label{J2g}
\mathbb{E} J_1^2(h+y_0^1,h+y_0^2,\lambda)
&=\frac{1}{2}\sum_{p=0}^{\tau-1}\frac{1}{n_p}\log\frac{b+\lambda-(d_{\tau}+(1-\eta t)d_{p+1}')}{b+\lambda-(d_{\tau}+(1-\eta t)d_p')}\\
&+\frac{1}{2}\sum_{p=\tau}^{k}\frac{1}{n_p}\log\frac{b+\lambda-d_{p+1}}{b+\lambda-d_p}.
\end{split}
\end{align}
Now recall from the statement of Proposition \ref{prop04}, $d(q)=\int_q^1\xi''(s)x(s)ds$ and $$\phi_u(q)=d(q_{\tau})+\frac{1-t}{1+t}(d(q)-d(q_{\tau})).$$ 
Since $d_{p}'=d(q_p)-d(q_{\tau})$ and
\begin{align*}
d_\tau+(1\pm\eta t)d_p'&=d(q_\tau)+\frac{1\pm \eta t}{1+t}(d(q_p)-d(q_\tau))=\left\{
\begin{array}{ll}
d(q_p),&\mbox{if $u\geq 0$},\\
\phi_u(q_p),&\mbox{if $u<0$,}
\end{array}
\right.
\end{align*}
we have
\begin{align}\label{eq4}
\frac{2h^2}{2(b\mp\lambda-(d_{\tau}+(1\pm\eta t)d_0'))}&=\left\{
\begin{array}{ll}
\frac{2h^2}{2(b\mp\lambda-d(0))},&\mbox{if $u\geq 0$},\\
\frac{2h^2}{2(b\mp\lambda-\phi_u(0))},&\mbox{if $u<0$.}
\end{array}
\right.
\end{align}
Using the fundamental theorem of calculus and the fact that $x(s)=m_p$ for $s\in [q_p,q_{p+1})$,
\begin{align}\label{eq5}
&\sum_{p=0}^{\tau-1}\frac{1}{n_p}\log\frac{b\mp\lambda-(d_\tau+(1\pm\eta t)d_{p+1}')}{b\mp\lambda-(d_\tau+(1\pm\eta t)d_p')}\notag\\
&=\sum_{p=0}^{\tau-1}\frac{1}{n_p}\left(\log(b\mp\lambda-(d_\tau+(1\pm\eta t)d_{p+1}'))-\log(b\mp\lambda-(d_\tau+(1\pm\eta t)d_p'))\right)\notag\\
&=\sum_{p=0}^{\tau-1}\frac{(1+t)}{m_p}\frac{(1\pm \eta t)}{1+t}\int_{q_p}^{q_{p+1}}\frac{\xi''(s)x(s)}{b\mp\lambda-(d_\tau+(1\pm\eta t)d_p')}ds\notag\\
&=\left\{
\begin{array}{ll}
(1\pm t)\int_{0}^{q_\tau}\frac{\xi''(s)}{b\mp\lambda-d(s)}ds,&\mbox{if $u\geq 0$},\\
(1\mp t)\int_{0}^{q_\tau}\frac{\xi''(s)}{b\mp\lambda-\phi_u(s)}ds,&\mbox{if $u<0$},
\end{array}
\right.
\end{align}
and
\begin{align}\label{eq6}
\sum_{p=\tau}^{k}\frac{1}{n_p}\log\frac{b\mp\lambda-d_{p+1}}{b\mp\lambda-d_{p}}&=\sum_{p=\tau}^{k}\frac{1}{n_p}\bigl(\log (b\mp\lambda-d(q_{p+1}))-\log(b\mp\lambda -d(q_{p}))\bigr)\notag\\
&=\sum_{p=\tau}^{k}\frac{1}{m_{p}}\int_{q_{p}}^{q_{p+1}}\frac{\xi''(s)x(s)}{b\mp\lambda-d(s)}ds\notag\\
&=\int_{q_{\tau}}^1\frac{\xi''(s)}{b\mp\lambda-d(s)}ds.
\end{align}
Plugging \eqref{eq4}, \eqref{eq5} and \eqref{eq6} into \eqref{J1g} and \eqref{J2g}, the equation \eqref{prop05}, Lemmas \ref{lem1}, \ref{prop05} and Proposition \ref{RSB} together complete our proof by taking $N\rightarrow\infty$ and $\alpha\downarrow 0$ in \eqref{lem1:eq1} and noting the usual large deviation principle $\lim_{N\rightarrow\infty}\tau_N^b=2^{-1}(b-1-\log b).$ 

\end{proof}

\section{Proofs of main results}

Now we are ready to prove our main results. Throughout this section, $({x},{b})$ stands for the optimizer in \eqref{pf}. Denote $ {d}(q) = \int_q^1\xi''(s)  {x}(s)ds$ and $$ {\phi}_u(q)= {d}(u)+\frac{1-t}{1+t}( {d}(q)- {d}(u)).
$$ 
First of all, we start with a proposition that is used to determine the value $u^*$ stated in Theorems \ref{thm0} and \ref{thm01}. 
Let $u_{ {x}}$ be the smallest value of the support of $ {x}.$ A crucial fact about $u_{ {x}}$ is that it must satisfy the following equation,
\begin{align}\label{eq8}
\frac{h^2+\xi'(u_{ {x}})}{( {b}- {d}(0))^2}=u_{ {x}}.
\end{align}
This can be seen from the proof of Theorem 7.2 in \cite{talagrand2006free}. In particular, \eqref{eq8} implies $u_{ {x}}>0$ if $h\neq 0.$

\begin{Proposition}\label{prop01}
For $t\in(0,1)$, define the function
\begin{align} \label{fixedpoint}
f(u)=\frac{h^2+t\xi'(u)}{( {b}- {d}(0))^2}-u
\end{align}
for $u\in [-u_{ {x}},u_{ {x}}].$ Then $f(u)=0$ has a unique solution $u^*$. Moreover, $u^*=0$ when $h =0$ and $u^* \in (0,u_{ {x}})$ when $h\neq 0$.
\end{Proposition}

\begin{proof} Note that $\xi'''$ is an odd function. This implies that $f$ is convex on $[0,u_{ {x}}]$ and is concave on $[-u_{ {x}},0]$. Assume that $h\neq 0.$ In this case, since $f(0)>0$ and $f(u_{ {x}})<0$ by \eqref{eq8}, the intermediate value theorem and the convexity of $f$ on $[0,u_{ {x}}]$ conclude that $f(u)=0$ has a unique solution $u^*$ on $[0,u_{ {x}}]$ and it satisfies $u^*\in(0,u_{ {x}}).$ In addition, since from \eqref{eq8},
$$
f(-u_{ {x}})=-\frac{-h^2+t\xi'(u_{ {x}})}{( {b}- {d}(0))^2}+u_{ {x}}>-\frac{h^2+\xi'(u_{ {x}})}{( {b}- {d}(0))^2}+u_{ {x}}=0,
$$
the concavity of $f$ on $[-u_{ {x}},0]$ and $f(0)>0$ imply that $f(u)=0$ has no solution on $[-u_{ {x}},0].$ This finishes the proof for the case $h\neq 0.$ The situation for $h=0$ is essentially identical. If $u_{ {x}}=0$, obviously $u^*=0.$ If $u_{ {x}}\neq 0,$ we still have $f(-u_{ {x}})>0>f(u_{ {x}})$, but now $f(0)=0.$ The convexity and concavity of $f$ on $[0,u_{ {x}}]$ and $[-u_{ {x}},0]$ respectively conclude that $0$ is the unique solution to $f(u)=0$ on $[-u_{ {x}},u_{ {x}}].$
\end{proof}

\begin{proof}[Proof of Theorem \ref{thm0}] Note that 
\begin{align}
\begin{split}\label{eq10}
\mathcal{P}_u( {x}, {b},0)&=\left\{
\begin{array}{ll}
T_u(x,b,0)+\frac{h^2}{ {b}- {d}(0)},&\mbox{if $u\in[0,1]$},\\
T_u(x,b,0)+\frac{h^2}{ {b}- {\phi}_{|u|}(0)},&\mbox{if $u\in[-1,0)$},
\end{array}\right.
\end{split}
\end{align}
where from \eqref{eq12},
\begin{align}
\begin{split}\label{eq9}
T_u(x,b,0)&:=\frac{1+t}{2}  \int_0^{|u|} \frac{\xi''(s)}{ {b}- {d}(s)} ds+ \frac{1-t}{2}  \int_0^{|u|} \frac{\xi''(s)}{ {b}- {\phi}_{|u|}(s)} ds\\
&\quad+\int_{|u|}^1 \frac{\xi''(s)}{ {b}- {d}(s)} ds+ {b}-1- \log  {b} - \int_0^1 \xi''(q)  {x}(q)dq.
\end{split}
\end{align}
 Consider first that $|u|>u_{ {x}}.$ Since $ {x}(q)>0$ for $q\in(u_{ {x}},|u|)$, we have that for all $s\in[0,|u|),$ 
\begin{align}\label{dsphi}
 {d}(s)- {\phi}_{|u|}(s) = \frac{2t}{1+t}( {d}(s)- {d}(|u|))=\int_s^{|u|} \xi''(q) {x}(q) dq >0
\end{align}
and from \eqref{eq10}, 
$$
\mathcal{P}_u( {x}, {b},0)\leq T_u(x,b,0)+\frac{h^2}{ {b}- {d}(0)}
$$
for any $u\in[-1,1]$. In addition, from \eqref{dsphi}, the first line of \eqref{eq9} is strictly bounded above by
\begin{align*}
\frac{1+t}{2} \int_0^{|u|} \frac{\xi''(s)}{ {b}- {d}(s)} ds+ \frac{1-t}{2}  \int_0^{|u|} \frac{\xi''(s)}{ {b}- {d}(s)} ds=\int_{0}^{|u|}\frac{\xi''(s)}{ {b}- {d}(s)} ds
\end{align*}
and as a result, these inequalities together with the equation \eqref{parisifunc} lead to $\mathcal{P}_u( {x}, {b},0)<2\mathcal{P}( {x}, {b})$. This completes the proof for \eqref{thm0:eq1} with $|u|>u_{ {x}}$ by using Proposition \ref{prop04}. As for the case $|u|\leq u_{ {x}},$ since $x(q)=0$ for $q\in[0,|u|),$ we have that for all $s\in[0,|u|],$
\begin{align*}
d(s)&=\int_{u_{ {x}}}^1\xi''(q) {x}(q)dq= {\phi}_{|u|}(s).
\end{align*} 
This allows us to write 
\begin{align*}
\mathcal{P}_u( {x}, {b},\lambda)&=\log \sqrt{\frac{b^2}{b^2-\lambda^2}}+\frac{h^2}{b-\lambda- d(0)}\\
&\quad+  \frac{1+t}{2}  \int_0^{|u|} \frac{\xi''(s)}{b-\eta\lambda-d(s)} ds+ \frac{1-t}{2}  \int_0^{|u|} \frac{\xi''(s)}{b+\eta\lambda-d(s)} ds\\
&\quad +\frac{1}{2} \int_{|u|}^1 \frac{\xi''(s)}{b-\lambda-d(s)} ds + \frac{1}{2} \int_{|u|}^1 \frac{\xi''(s)}{b+\lambda-d(s)} ds \\
 &\quad-\lambda u +b-1- \log b - \int_0^1 \xi''(q) x(q)dq
\end{align*}
for all $u\in [-u_{ {x}},u_{ {x}}].$ A direct computation gives that
\begin{align*}
\mathcal{P}_u( {x}, {b},0)&=2\mathcal{P}( {x}, {b}),\,\,\partial_\lambda\mathcal{P}_u( {x}, {b},0)=f(u)
\end{align*}
and moreover, for $\lambda$ in a small open neighborhood of $0$, 
\begin{align*}
|\partial_{\lambda\lambda}\mathcal{P}_u( {x}, {b},\lambda)|&\leq L,
\end{align*}
where $L$ is a positive constant independent of $\lambda.$
Consequently, applying the Taylor theorem and taking $\lambda =-\delta f(u)/L$ for sufficiently small $\delta>0$, if $u\in[-u_{ {x}},u_{ {x}}]$ and $u\neq u^*,$ then Proposition \ref{prop01} yields
\begin{align*}
\limsup_{\alpha\downarrow 0}\limsup_{N\rightarrow\infty}p_{N,u,\alpha}
&\leq\mathcal{P}_u( {x}, {b},0)+\partial_\lambda\mathcal{P}_u( {x}, {b},0) \lambda+\frac{L}{2}\lambda^2\\
&=2\mathcal{P}( {x}, {b})-\frac{\delta f(u)^2}{L}\Bigl(1-\frac{\delta}{2}\Bigr)\\
&<2\mathcal{P}( {x}, {b}).
\end{align*}
This proves \eqref{thm0:eq1} for $|u|\leq u_{ {x}}$ with $u\neq u^*.$
\end{proof}

At the end of this section, we prove Theorem \ref{thm01}. It will need an inequality of Gaussian concentration of measure from the appendix of \cite{panchenko2007note} stated below.

\begin{Lemma}\label{lem03}
Let $\nu$ be a finite measure on $\mathbb{R}^N$ and $g(z)$ be a centered Gaussian process indexed on $\mathbb{R}^N$ such that
$\mathbb{E} g(z)^2 \leq a$ for $z$ in the support of the measure $\nu$. If
\begin{align*}
X = \log \int_{\mathbb{R}^N} \exp g(z) d \nu (z),
\end{align*}
then for all $s>0$,
\begin{align*}
\mathbb{P} (|X-\mathbb{E}X| \geq s) \leq 2 \exp \left( -\frac{s^2}{4a} \right).
\end{align*}
\end{Lemma}

\begin{proof}[Proof of Theorem \ref{thm01}] Let $\varepsilon>0$ and set $S_\varepsilon=[-1,1]\setminus (u^*-\varepsilon,u^*+\varepsilon).$ Using \eqref{pf} and \eqref{thm0:eq1}, for any $u\in S_\varepsilon$, there exist $\alpha_u>0$ and $N_u\in\mathbb{N}$ such that
$
p_{N,u,\alpha_u}<2p_N-\varepsilon
$
for all $N\geq N_u$. Set $I_u=(u-\alpha_u,u+\alpha_u)$. Since $\{I_u:u\in S_\varepsilon\}$ forms an open covering of $S_\varepsilon$, the compactness of $S_\varepsilon$ implies that there exist $u_1,\ldots,u_n$ such that $\cup_{i=1}^nI_{u_i}$ covers $S_{\varepsilon}.$ Letting $N_0=\max\{N_{u_i}:1\leq i\leq n\}$, we obtain that 
\begin{align}
\label{eq11}
p_{N,u_i,\alpha_{u_i}}<2p_N-\varepsilon
\end{align}
for all $1\leq i\leq n$ and $N\geq N_0.$ Next, applying Lemma \ref{lem03}, the event $A_N$ such that 
\begin{align*}
\min\left\{\frac{1}{N}\log Z_N^1,\frac{1}{N}\log Z_N^2\right\}\geq p_N-\frac{\varepsilon}{8}
\end{align*}
and
\begin{align*}
\max\left\{\frac{1}{N}\log Z_{N,u_i,\alpha_{u_i}}:i=1,\ldots,n\right\}\leq p_{N,u,\alpha}+\frac{\varepsilon}{8}
\end{align*}
has probability at least $1-K\exp(-N/K)$, where $K>0$ is independent of $N.$ On $A_N$, these inequalities and \eqref{eq11} lead to
\begin{align*}
\la 1_{\left\{R_{1,2}\in S_\varepsilon\right\}}\ra&\leq \sum_{i=1}^n\bigl\la 1_{\{R_{1,2}\in I_{u_i}\}}\bigr\ra\\
&\leq \sum_{i=1}^n\exp\bigl(\log Z_{N,u_i,\alpha_{u_i}}-\log Z_{N}^1-\log Z_N^2\bigr)\\
&\leq \sum_{i=1}^n\exp\Bigl(-N\Bigl(p_{N,u_i,\alpha_{u_i}}-2p_N+\frac{\varepsilon}{4}\Bigr)\Bigr)\\
&\leq n\exp\Bigl(-\frac{3\varepsilon N}{4}\Bigr).
\end{align*}
Therefore,
\begin{align*}
\e\la 1_{\left\{R_{1,2}\in S_\varepsilon\right\}}\ra&\leq n\exp\Bigl(-\frac{3\varepsilon N}{4}\Bigr)\p(A_N)+\p(A_N^c)\\
&\leq n\exp\Bigl(-\frac{3\varepsilon N}{4}\Bigr)+K\exp\Bigl(-\frac{N}{K}\Bigr),
\end{align*}
which completes our proof.
\end{proof}


\thebibliography{99}


\bibitem{chatterjee2009disorder}
Chatterjee, S. (2009) Disorder chaos and multiple valleys in spin glasses. Preprint available at arXiv:0907.3381.

\bibitem{chen2013disorder}
Chen, W.-K. (2013) Disorder chaos in the {S}herrington-{K}irkpatrick model with external field. {\it Ann. Probab.}, {\bf 41}, no. 5, 3345--3391.

\bibitem{Chen12}
Chen, W.-K. (2013) The {A}izenman-{S}ims-{S}tarr scheme and {P}arisi formula for mixed $p$-spin spherical models. {\it Electron. J. Probab.}, {\bf 18},
no. 94, 1--14.

\bibitem{chen2014chaos}
Chen, W.-K. (2014) Chaos in the mixed even-spin models. {\it Comm. Math. Phys.}, {\bf 328}, no. 3, 867--901.

\bibitem{chen2013approach}
Chen, W.-K. and Panchenko, D. (2013) An approach to chaos in some mixed $p$-spin models. {\it Probab. Theory Rel. Fields}, {\bf 157}, no. 1-2, 389--404.

\bibitem{crisanti1993sphericalp}
Crisanti, H. and Sommers, H. J. (1993) The spherical $p$-spin interaction spin-glass model. {\it Z. Phys. B. condensed Matter}, {\bf 92}, no. 2, 257--271.

\bibitem{panchenko2007note}
Panchenko, D. (2007) A note on Talagrand's positivity principle. {\it Electron. Comm. Probab.}, {\bf 12}, 401--410.
  
\bibitem{panchenko2007overlap}
Panchenko, D. and Talagrand, M. (2007) On the overlap in the multiple spherical SK models. {\it Ann. Probab.}, {\bf 35}, no. 6, 2321--2355.


\bibitem{rizzo07}
Rizzo, T. (2009) Spin Glasses: Statics and Dynamics: Summer School, Paris 2007. {\it Progr. Probab.}, {\bf 62}, {143--157}.


\bibitem{talagrand2006free}
Talagrand, M. (2006) Free energy of the spherical mean field model. {\it Probab. Theory Rel. Fields}, {\bf 134}, no. 3, 339--382.

\bibitem{talagrand2011meanII}
Talagrand, M. (2011) Mean field models for spin glasses. Ergebnisse der Mathematik und ihrer Grenzgebiete. 3. Folge. A Series of Modern Surveys in Mathematics, {\bf 55}, Springer, Berlin.
\end{document}